\documentclass{amsart}

\usepackage{amsmath,amssymb}
\usepackage[dvipsnames]{xcolor}
\usepackage[colorlinks,citecolor=orange,linkcolor=red!90!black]{hyperref} 
\usepackage{tikz,graphicx}
\usepackage{aligned-overset} 
\usepackage[shortlabels]{enumitem}
\usepackage{mathtools}
\usepackage{booktabs} 

\newtheorem{theorem}{Theorem}
\newtheorem{lemma}[theorem]{Lemma}
\newtheorem{corollary}[theorem]{Corollary}
\newtheorem{proposition}[theorem]{Proposition}
\theoremstyle{definition}
\newtheorem{remark}[theorem]{Remark}
\newtheorem{definition}[theorem]{Definition}

\numberwithin{theorem}{section}
\numberwithin{equation}{section}

\newcommand{\Z}{\mathbb{Z}} 
\newcommand{\R}{\mathbb{R}} 
\newcommand{\T}{\mathbb{T}^d} 
\newcommand{\intersect}{\cap} 
\newcommand{\grad}{\nabla} 
\newcommand{\abs}[1]{\left| #1 \right|} 
\newcommand{\normX}[1]{{\left\lVert #1 \right\rVert}_X} 
\newcommand{\normV}[1]{{\left\lVert #1 \right\rVert}_{\text{sup}}} 
\newcommand{\normC}[1]{{\left\lVert #1 \right\rVert}_{C^0(\T)}} 

\begin{document}
\title[Nonlinear Fokker-Planck]{Global Well-Posedness of a Nonlinear Fokker-Planck Type Model of Grain Growth}

\author{Batuhan Bayir}
\address[Batuhan Bayir]
{Department of Mathematics,
The University of Utah,
Salt Lake City, UT 84112, USA}
\email{bayir@math.utah.edu}

\author{Yekaterina Epshteyn}
\address[Yekaterina Epshteyn]
{Department of Mathematics,
The University of Utah,
Salt Lake City, UT 84112, USA}
\email{epshteyn@math.utah.edu}

\author{William M Feldman}
\address[William M Feldman]
{Department of Mathematics,
The University of Utah,
Salt Lake City, UT 84112, USA}
\email{feldman@math.utah.edu}

\keywords{polycrystalline materials, microstructures, grain growth modeling, dissipation principle, nonlinear Fokker-Planck equation, global well-posedness}
\subjclass{74A25, 
74N15, 
35Q84, 
35A01  
}

\begin{abstract}
Most technologically useful materials spanning multiple length scales are polycrystalline. Polycrystalline microstructures are composed of a myriad of small crystals or grains with different lattice orientations which are separated by interfaces or grain boundaries. The changes in the grain and grain boundary structure of polycrystals highly influence the material’s properties including, but not limited to, electrical, mechanical, and thermal. Thus, an understanding of how microstructures evolve is essential for the engineering of new materials. In this paper, we consider a recently introduced nonlinear Fokker–Planck-type system and establish a global well-posedness result for it. Such systems under specific energy laws emerge in the modeling of the grain boundary dynamics in polycrystals.
\end{abstract}

\maketitle

\section{Introduction}
In this paper, we study the global well-posedness of a nonlinear Fokker-Planck type model of grain growth introduced in \cite{epshteyn2023local}. Grain growth is a highly complex multiscale-multiphysics process appearing in materials science which describes the evolution of the microstructure of polycrystalline materials, e.g. \cite{BHATTACHARYA1996529,DK:gbphysrev,BobKohn,DK:DCDSB,MR3729587,epshteyn-timescales,patrick2023relative,paperRickman,LIU2023105329,barmak2024advances,qiu2025grain}. These materials consist of many small monocrystalline grains which are separated by interfaces or grain boundaries as depicted in Figure \ref{coarsening}. 

\begin{figure}
    \centering
    \scalebox{1}{
    \begin{tikzpicture}
    \def\a{2.5}  
    \def\b{2}    
    \definecolor{custom-blue}{HTML}{413c97} 
                                                                                                  
    \draw[thick,custom-blue] (0,0) ellipse (\a cm and \b cm);

    \node[anchor=west,custom-blue] at (-0.95,-0.10) {$\boldsymbol{a}(t)$};
    
    \node[anchor=west] at (-2.4,-0.1) {$\alpha^{(1)}$};
    \node[anchor=west] at (0.8,1.34) {$\alpha^{(2)}$};
    \node[anchor=west] at (1.35,-0.50) {$\alpha^{(3)}$};
    
    \foreach \angle in {40, 127, 257} {
        \pgfmathsetmacro{\x}{\a * cos(\angle)}
        \pgfmathsetmacro{\y}{\b * sin(\angle)}
        \draw[thick,custom-blue] (0,0) -- (\x,\y);
    }

    \node[anchor=west] at (-1.93,1.1) {$\alpha^{(2)}-\alpha^{(1)}$};
    \node[anchor=west] at (-0.2,0.55) {$\alpha^{(3)}-\alpha^{(2)}$};
    \node[anchor=west] at (-1.36,-1.4) {$\alpha^{(1)}-\alpha^{(3)}$};

    \def\gridsize{0.20}  
    \foreach \x/\y/\angle in {0.65/-1/15, 0.3/0.8/35, -1.05/-1/70} {
        \begin{scope}[shift={(\x,\y)}, rotate=\angle]
            \foreach \i in {1, 2, 3} {
                \draw[thick] (0, \i*\gridsize) -- (4*\gridsize, \i*\gridsize);
            }
            \foreach \i in {1, 2, 3} {
                \draw[thick] (\i*\gridsize, 0) -- (\i*\gridsize, 4*\gridsize);
            }
        \end{scope}
    }

    \fill[custom-blue] (0,0) circle (3pt);
    \end{tikzpicture}
    }
    \includegraphics[width=0.395\linewidth]{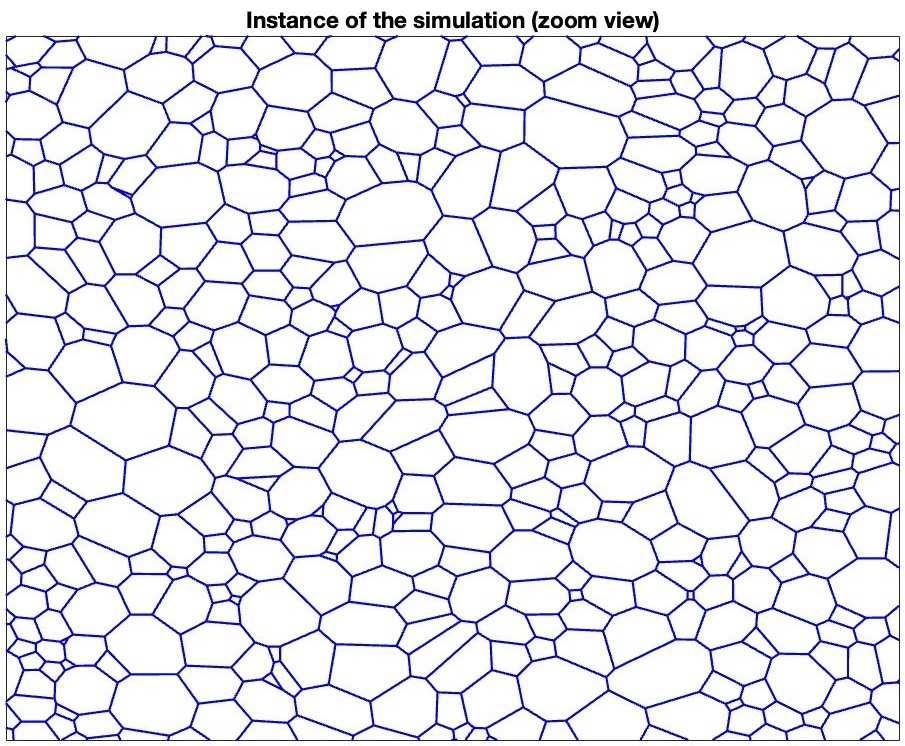}
    \caption{
    Left figure: A schematic plot of three grain boundaries that meet at a triple junction point $\boldsymbol{a}(t)$. For each $i,j \in \{1,2,3\}$, $\alpha^{(j)}=\alpha^{(j)}(t)$ represents the lattice orientation that corresponds to grid lines on the figure and the differences $\alpha^{(i)}-\alpha^{(j)}$ represent the lattice misorientations. Right figure: An instance of the simulation of the two-dimensional grain network that is a collection of grain boundaries and triple junction points \cite{epshteyn2021motion,epshteyn-timescales}. 
    }
    \label{coarsening}
\end{figure}
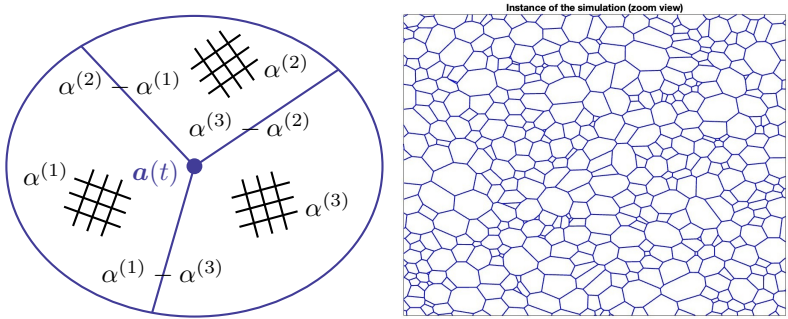
The properties of the resulting material depend in a complicated way on many factors, including, but not limited to, initial microstructure configuration and annealing process. By changing the grain statistics of a material, it is possible to control, for instance,  its mechanical, thermal, and electrical properties \cite{kurtz1980microstructure,barmak2024advances}. In particular, smaller grains lead to improvement in material strength and toughness, which are crucial in the ceramics industry \cite{rice1994hardness}. On the other hand, larger grains usually help to reduce electrical resistivity which is useful in the thin film and microchip industries \cite{thompson1994grain,baurer2010linking,barmak2024advances}. So, for manufacturing purposes, it is highly valuable to predict and control polycrystalline microstructure.

For this aim, in \cite{epshteyn2021motion}, the authors introduce a new system for the evolution of the two-dimensional grain boundary network with finite mobility of the triple junctions (triple junction drag) and with dynamic lattice orientations/misorientations. In \cite{epshteyn2021motion}
the grain boundary curvature was assumed to be the fastest time scale and was relaxed first, hence, the grain boundaries became line segments. A relevant recent work \cite{qiu2025grain} provides the experimental evidence justifying such relaxation. The work \cite{doi:10.1137/140999232} is also relevant. Thus, such model \cite{epshteyn2021motion}, which we call the \emph{vertex model}, tracks the evolution of lattice orientations $\alpha^{(j)}(t) \in \R$, $j$ as in the description of Figure \ref{coarsening} and triple junction points $\boldsymbol{a}(t) \in \R^2$. In technical terms this is a system of $O(N)$ ODEs, where $N$ is the number of grains.

Since polycrystalline materials in applications consist of $N \gtrsim 10^4$ monocrystalline grains (see, e.g., \cite{patrick2023relative}, but it also depends on the particular setting), this is a coupled system of a very large number of equations.  As is common in such multi-scale physics problems it is desirable to derive and study a macroscopic model which describes the evolution of statistical quantities related to the underlying microscopic system.  For this aim, the authors in \cite{epshteyn2022nonlinear,epshteyn2023local,epshteyn2024longtimeasymptoticbehaviornonlinear} studied the $N \to \infty$ limit of the vertex model resulting in the following free energy with inhomogeneous absolute temperature $D(x)$:\noindent 
\begin{equation}\label{FE}
    F[f] = \int_{\T} \left( D(x) f(x,t) (\log f(x,t)-1) + \phi(x) f(x,t) \right) dx
\end{equation}
together with the dissipation relation
\begin{equation}\label{dissipation-relation}
    \frac{d}{dt} F[f] = - \int_{\T} \frac{f}{\pi(x,t)} \abs{\grad(D(x)\log f + \phi(x))}^2 dx.
\end{equation}
Here, $\T \coloneqq \R^d / \Z^d$ is $d$-dimensional torus, and we assume that  $d \geq 1$ can be any dimension, $\pi(x,t)>0$ is mobility function, $D(x)>0$ is inhomogeneous but time-independent absolute temperature (it can
be viewed as a function of the fluctuation parameters of the lattice misorientations and of the
position of the triple junctions due to fluctuation-dissipation principle and it also accounts
for some information of the under-resolved mechanisms in the system, for example, such as topological changes during microstructure evolution \cite{Katya-Chun-Mzn3}), and $\phi(x)$ is the energy density of a grain boundary. In relation to the microscopic system described above, the state variable $x$ is a pair $(\boldsymbol{a},\Delta \alpha)$ where $\boldsymbol{a} \in \mathbb{T}^2$ is the triple junction location and $\Delta \alpha \in \mathbb{T}^2$ is the lattice misorientation. Then $f(x,t)$ is the time dependent probability density function on state space of triple junction and lattice misorientation pairs.

By the energetic-variational approach, see e.g., \cite{DK:gbphysrev,giga2017variational, Katya-Chun-Mzn3,epshteyn2023local}, the free energy \eqref{FE} and the dissipation relation \eqref{dissipation-relation} determine the following PDE:
\begin{equation}
    \tag{nFP-D}
    \label{FP-div-form}
    \left\{
    \begin{aligned}
    & \frac{\partial f}{\partial t} = \grad \cdot \left( \frac{f}{\pi(x,t)} \grad \left( D(x) \log f + \phi(x) \right) \right), &\quad &x \in \T,\   t>0, \\
    & f(x,0)=f_0(x), &\quad &x \in \T.
    \end{aligned}   
    \right.
\end{equation}
Due to inhomogeneity of the absolute temperature $D(x)$, the system \eqref{FP-div-form} has a nonlinear term at first order, and is thus only semi-linear. Such inhomogeneity and the resulting non-linearity are very different from the vast existing literature on the Fokker-Planck type models and introduces several challenges for the mathematical analysis of the model, including for the global well-posedness study which we will resolve in this work. In addition to energetic-variational approach, the detailed derivation of this model from a kinematic continuity equations perspective is given in \cite{epshteyn2023local}.

Besides grain growth modeling, Fokker-Planck type partial differential equations also appear in many other areas of science and mathematics, ranging from physics to economics. Some particular examples are: modeling of a gene regulatory network in biology \cite{MR4196904}, cell migration \cite{cell-migration}, plasma physics \cite{plasma-physics}, gas flows \cite{gas-flows}, financial market dynamics \cite{financial-market-dynamics}, computational neuroscience \cite{computational-neuroscience}, galactic nuclei \cite{astrophysics}, probability theory \cite{barbu2024nonlinear}, and optimal transport \cite{jko}.

The contribution of this paper is to obtain the existence of the unique classical positive global solution of the nonlinear Fokker-Planck type equation \eqref{FP-div-form}. We achieve this result under some mild assumptions including Hölder regularity of the coefficients of the parabolic operator in \eqref{parabolic-operator} below and strict positivity of the initial data $f_0$ in \eqref{FP}. The detailed description of the assumptions can be found in Section \ref{assumptions} and the main results are in Theorem \ref{local-well-posedness} and Theorem \ref{global-well-posedness}.

In \cite{epshteyn2023local}, the authors obtain the unique classical local solution of \eqref{FP-div-form} under natural (no flux) boundary conditions. They assume $C^{2+\beta}_x$ regularity of the initial data for local existence, which is challenging to control via a priori estimates and to extend to a global solution. In contrast our work only uses uniform upper and lower bounds on the initial data to obtain local existence, and thus can be combined with a maximum principle \cite{epshteyn2022nonlinear,epshteyn2024longtimeasymptoticbehaviornonlinear} to continue the solution.  Also the result in \cite{epshteyn2023local} comes after the application of three different change of variables.
Our analysis, on the other hand, does not rely on any change of variables. 
Therefore, in the current paper,  we construct a direct and much simpler proof/method through the use of classical tools such as Gaussian bounds and Schauder estimates, given in Theorems \ref{gaussiann-bounds} and \ref{schauder-estimates} for the readers' convenience. Such direct method gives better insights in the structure of the model under study.
Moreover, the ideas developed in this work could be extended and applied to more general class of nonlinear Fokker-Planck type systems of interest, beyond the context of grain growth modeling, see Remark \ref{generalization-of-FP}.

In \cite{epshteyn2022nonlinear,epshteyn2024longtimeasymptoticbehaviornonlinear} the authors study long time asymptotic behavior of the solutions of \eqref{FP-div-form}. Call $f^{\textrm{eq}}(x)$ to be the energy minimizing equilibrium state of \eqref{FP-div-form}, which takes the following form: 
\begin{equation}\label{eq-state}
    f^{\mathrm{eq}}(x) \coloneqq \exp \left( -\frac{\phi(x)-C^{\textrm{eq}}}{D(x)} \right)
\end{equation}
where $C^{\textrm{eq}} \in \R$ is a free parameter determined by the total mass. The works \cite{epshteyn2022nonlinear,epshteyn2024longtimeasymptoticbehaviornonlinear} show an a priori estimate bounding the maximum and minimum of solutions of \eqref{FP-div-form} in terms of $D(x)$, $f_0(x)$, and $f^{\textrm{eq}}(x)$, and show convergence to equilibrium as $t \to \infty$. All of these results, however, are conditional on the existence of a global classical solution. Our result complements and completes the picture in \cite{epshteyn2022nonlinear,epshteyn2023local,epshteyn2024longtimeasymptoticbehaviornonlinear} by showing the existence of such solutions. 

\textbf{Summary of the proof.} 
Our analysis below employs tools of linear parabolic PDEs and fixed point theory. We remark that in the existence theory for the semi-linear equations via the fixed point approach, it helps to separate the lower order nonlinearity and treat it as a source term. Hence, from the model \eqref{FP-div-form}, we obtain the following expanded form, which decomposes linear and nonlinear terms:
\begin{equation}
    \tag{nFP}
    \label{FP}
    \left\{
    \begin{aligned}
    & \frac{\partial f}{\partial t} = L_{\textup{FP}}f + \grad \cdot \left( \frac{\grad D (x)}{\pi (x,t)}f \log f \right), &\quad &x \in \T,\   t>0, \\
    & f(x,0)=f_0(x), &\quad &x \in \T.
    \end{aligned}   
    \right.
\end{equation}
Here $L_{\textup{FP}}$ is a divergence form linear operator as defined below:
\begin{equation}\label{parabolic-operator}\tag{LO}
    L_{\textup{FP}}f \coloneqq  \grad \cdot \left( \frac{D(x)}{\pi (x,t)} \grad f \right) + \frac{\grad \phi(x)}{\pi (x,t)}\cdot \grad f + \grad \cdot \left( \frac{\grad \phi (x)}{\pi (x,t)} \right) f.
\end{equation}
%
The linear parabolic non-homogeneous equation \eqref{LPNE} contains the nonlinearity of \eqref{FP} as a source term now.
After that, in Definition \ref{map-def}, by linearity and Duhamel's principle, we present corresponding map of \eqref{LPNE} in terms of the fundamental solution of $\partial_t-L_{\textup{FP}}$, which is \eqref{map}. Using Gaussian bounds of $\partial_t-L_{\textup{FP}}$ from Corollary \ref{gaussian-integral-estimates}, we prove that the map \eqref{map} is contraction mapping and obtain the unique fixed point solution of \eqref{LPNE} which is also solution of \eqref{FP}. Applying Schauder estimates of Proposition \ref{schauder-estimates}, we upgrade the regularity of the fixed point solution and obtain the unique classical local solution. To extend to a global solution, we repeat our argument on nested time intervals via induction together with an a-priori estimate proved in \cite{epshteyn2022nonlinear,epshteyn2024longtimeasymptoticbehaviornonlinear}, see Proposition \ref{a-priori-estimate-1}.

\textbf{Organization of the paper.} In Section \ref{preliminaries-section}, we state some helpful notations, our assumptions, and some preliminary results from parabolic PDE theory. In Section \ref{local-existence-section}, we establish local existence of solutions. Lastly, in Section \ref{global-existence-section}, we obtain the unique positive classical global solution of \eqref{FP} and complete our discussion.

\section{Preliminaries}
\label{preliminaries-section}
In this section, we list some useful notations, our assumptions, and fundamental results related to parabolic PDEs. 
In the table below, we outline some basic notations
  which will be used throughout the paper.

\subsection{Notations}\label{notation}
\begin{table}[!hbtp]
    \begin{tabular}{@{}ll@{}}
    \toprule
    Function & Description \\ \midrule
    $f(x,t)$ & Unknown probability density function \\
    $f_0(x)$ & Initial condition \\
    $f^\textup{eq}(x)$ & Energy minimizing equilibrium state defined in \eqref{eq-state}\\
    $\pi(x,t)$ & Mobility \\
    $D(x)$ & Inhomogeneous absolute temperature coefficient \\
    $\phi(x)$ & Grain boundary energy density \\
    $V(x,t)$ & Coefficient of the nonlinear part defined by $V(x,t) \coloneqq \frac{\grad D(x)}{\pi(x,t)}$  \\ \bottomrule
    \end{tabular}
    \caption{Summary of the functions that appear in \eqref{FP}. \\ 
    Assumptions on these functions can be found in Section \ref{assumptions}.}
    \label{variables-table}
\end{table}
\begin{enumerate}[label = (N\arabic*)]
    \item \label{notation-constants} Constants $c,C>0$ may change their values from line to line.
    \item \label{notation-abs-norm} We use the symbol $\abs{\cdot}$
      for Euclidean distance on $\R^d$ and the distance on $\T$
      inherited from the Euclidean distance on $\R^d$ via the isometry
      with $\R^d / \Z^d$. Here, $\mathbb{Z}^d$ is
        $d$-dimensional integer lattice and, in this paper, $d$ is any $d \geq 1$.
\end{enumerate}

\subsection{Assumptions}\label{assumptions}
We make the following positivity and regularity assumptions on $L_{\textup{FP}}$ defined in \eqref{parabolic-operator} and $f_0$:
\begin{center}
    \begin{enumerate}[label = (A\arabic*)]
        \item \label{assumption-constants} $\frac{D(x)}{\pi (x,t)} \geq \theta$ for all $(x,t) \in \T \times [0,\infty)$ for some arbitrary $\theta>0$ which is independent from $x$ and $t$. With this assumption, the linear part of \eqref{FP}, $\partial_t-L_{\textup{FP}}$, defines a linear second-order uniformly parabolic differential operator in the sense of Definition \ref{parabolic-operator-def}. Here, $\T$ is $d$-dimensional torus.
        \item \label{regularity-of-coefficients} The coefficients of $L_{\textup{FP}}$ when they are considered in non-divergence form: 
        $\frac{D(x)}{\pi(x,t)}$, $\frac{\grad \phi(x)}{\pi(x,t)}+\grad \left(\frac{D(x)}{\pi(x,t)}\right)$, and $\grad \cdot \left( \frac{\grad \phi (x)}{\pi (x,t)} \right)$ are bounded and belong to $C^{{1+\beta},{\beta/2}}_{x,t}(\T \times [0,\infty))$. In addition, the coefficient of the nonlinear part, $\frac{\grad D (x)}{\pi (x,t)}$ satisfies the same assumptions. Here, $\beta \in (0,1)$ is the Hölder exponent of the coefficients.
        \item  \label{assumption-initial-data} $\Lambda \geq f_0(x) \geq 4\mu$ for some arbitrary but fixed $\Lambda,\mu > 0$ which are independent from $x \in \T$. 
        \item\label{assumption-boundedness} There is $C_D\geq 1$ which is independent from $x \in \T$ such that $D(x) \geq C_D$. There are also $C_{\pi}^{\textup{low}},C_{\pi}^{\textup{up}} >0$ which are independent from $(x,t) \in \T \times [0,\infty)$ such that $C_{\pi}^{\textup{low}} \leq \pi(x,t) \leq C_{\pi}^{\textup{up}}$ \cite{epshteyn2024longtimeasymptoticbehaviornonlinear}.
    \end{enumerate}
\end{center}

\subsection{Parabolic equations}
For ease of reference and for readers' convenience, in this subsection, we review some essential facts from the linear parabolic PDEs. The main source is Friedman's classical book \cite{friedman2008partial}, as well as \cite{ladyzhenskaya,lieberman1996second}. This subsection could be skipped on the first reading and referred back to as needed.

\begin{definition}[Parabolic operators]\label{parabolic-operator-def}
    Consider differential operators in the form:
\begin{equation*}
     L f \coloneqq \sum_{i,j=1}^{d} \frac{\partial}{\partial x_i} \left( a^{ij}(x,t) \frac{\partial f}{\partial x_j} \right) + \sum_{i=1}^{d}b^{i}(x,t) \frac{\partial f}{\partial x_i}+c(x,t)f.
\end{equation*}
We call $\partial_t-L$ to be a second-order uniformly parabolic operator in divergence form if it satisfies the following two properties:
\begin{enumerate}[(P1)]
    \item $a^{ij}(x,t)=a^{ji}(x,t)$ for all $i,j \in \{1,\dots,d\}$ and $(x,t) \in \T \times (0,T]$,
    \item $\sum_{i,j=1}^{d} a^{ij}(x,t) \xi_i \xi_j \geq \lambda \abs{\xi}^2$ holds for all $(x,t) \in \T \times (0,T]$ where $\lambda > 0$ is uniform parabolicity constant which is independent from $(x,t)$.
\end{enumerate}
\end{definition}

Observe that, with Assumption \ref{assumption-constants}, $\partial_t - L_{\textup{FP}}$ defines a second-order uniformly parabolic operator in divergence form where $L_{\textup{FP}}$ specified as in \eqref{parabolic-operator}.

\begin{remark}
    We always assume that the coefficients of $\partial_t-L$ are sufficiently differentiable when we switch to the non-divergence form of the operator by applying product rule. This is convenient for applying the results of Friedman \cite{friedman2008partial}, which are stated for operators in non-divergence form.
\end{remark}

Next we introduce the fundamental solutions of the operator $\partial_t - L$ on $\T \times [0,T]$, which play the role of the heat kernel in this more general context.

\begin{definition}[{Fundamental solution \cite[p.~3]{friedman2008partial}}]\label{fundamental-solution}
    Let $\partial_t - L$ be a second order uniformly parabolic differential operator. We call $K(x,t;y,s)$ to be the fundamental solution of $\partial_t - L$ in $\T \times [0,T]$ if the following two properties hold:
    \begin{enumerate}[(P1)]
        \item $K(x,t;y,s)$ is solution of $\partial_t h = Lh$ as a function of $(x,t) \in \T \times [0,T]$ for every fixed $(y,s)$ with $y \in \T$ and $s<t\leq T$,
        \item for every continuous function $f_0$ on $\T$, $K$ satisfies the limit relation given below
            \begin{equation*}
                \lim_{t \rightarrow s} \int_{\T} K(x,t;y,s) f_0(y) dy = f_0(x).
            \end{equation*}
    \end{enumerate}
\end{definition}

The fundamental solution $K$ is strictly positive function \cite[Thm.~11,~Ch.~2]{friedman2008partial}: 
\begin{equation}\label{positivity-fund-sol}
    K(x,t;y,s)>0 \ \hbox{ for all } \ x,y \in \T, \ \text{and } t>s.
\end{equation}

By using the fundamental solution and Duhamel's principle, it is possible to come up with the following representation formula for the initial value problem associated with $\partial_t-L$.
\begin{theorem}[{Representation formula \cite[Thm.~10,~Ch.~3]{friedman2008partial}}]\label{rep-formula}
    Let $\partial_t-L$ be a second-order uniformly parabolic operator in divergence form such that all of its coefficients in non-divergence form are $\beta$-Hölder continuous. Then, there is a $C^{2+\beta,1+\beta/2}_{x,t}$ classical solution of $\partial_t h = Lh$ with the initial condition $f_0 \in C^0 (\T)$ which
    can be represented by 
    \begin{equation*}
        h(x,t) = \int_{\T} K(x,t;y,0) f_0(y) dy 
    \end{equation*}
    where K is the fundamental solution of $\partial_t - L$ as in Definition \ref{fundamental-solution}. 
\end{theorem}

Moreover, the fundamental solution $K$ of divergence form operator $\partial_t-L$ satisfies the following comparison principle:
\begin{equation}\label{single-int-fund-bound}
    \exp \left( c_{\textup{inf}}(t-s) \right) \leq \int_{\T} {K(x,t;y,s)}dy \leq \exp\left( c_{\textup{sup}}(t-s) \right)
\end{equation}
where $c_{\textup{inf}}$ and $c_{\textup{sup}}$ denotes the infimum and supremum of the coefficient $c(x,t)$ of $\partial_t-L$ on $\T \times (s,\infty)$ respectively and $c(x,t)$ as in  Definition \ref{parabolic-operator-def}. This follows from a standard parabolic comparison principle because the left and right hand side of \eqref{single-int-fund-bound} are, respectively, a subsolution and supersolution of $(\partial_t - L)\phi =0$, the middle term $v(x,t) \coloneqq \int_{\T} {K(x,t;y,s)}dy$ solves $(\partial_t - L)v = 0$, and all three have the same initial data, identically $1$. 
In particular, if $c(x,t)=0$, then \eqref{single-int-fund-bound} reduces to the maximum principle.

Such fundamental solutions and their derivatives satisfy Gaussian
bounds. In other words,  under suitable regularity assumptions
  on the coefficients of the operator $\partial_t-L$, the
  corresponding fundamental solution $K$ satisfies estimates matching
  those for the heat kernel, see detailed result in Theorem
  \ref{gaussiann-bounds} below. Note, that we will only use the unit
  time bounds, the long time bounds are of a different form, since
  after a unit time, the fundamental solution can be affected by the topology of the torus.
\begin{theorem}[{Gaussian bounds on the fundamental solution \cite[Thm.~7,~Ch.~9]{friedman2008partial}, \cite[Thm.~8,~Ch.~9]{friedman2008partial}, 
\cite[p.~255]{friedman2008partial}} (adapted)]\label{gaussiann-bounds}
    Let $K$ be the fundamental solution of the second-order uniformly parabolic operator $\partial_t-L$ in non-divergence form such that all of its coefficients are bounded and $C^{l+\beta,0}_{x,t}$-Hölder regular. 
    Then, for all $x,y \in \T$ and $t-s\leq1$, the following Gaussian bounds hold on $K$ and its derivatives:
    \begin{equation}\label{gaussian-bound-time-space}
        \abs{\partial_t^{a} \grad_y^{b} K(x,t;y,s)} \leq C (t-s)^{-(d+2a+b)/{2}} \exp \left( \frac{-c \abs{x-y}^{2}}{t-s} \right)
    \end{equation}
    and
    \begin{align}
    \begin{split}
        \label{gaussian-gradient-bound-holder-type}
        |\grad_y K&(x,t;y,s) - \grad_y K(x',t;y,s)| \\
        &\leq C \abs{x-x'}^{\beta}(t-s)^{-(d+1+\beta)/2}\left( \exp \left(\frac{-c\abs{x-y}^2}{t-s}\right) + \exp \left(\frac{-c\abs{x'-y}^2}{t-s}\right) \right)
    \end{split}
    \end{align}
    where $\beta \in (0,1)$ and $c,C > 0$ are constants which depend on the dimension, uniform parabolicity constant $\lambda$ of $L$, uniform upper bound of the coefficients of $L$, Hölder regularity of the coefficients of $L$, and order of the derivatives $a,b \in \Z_{\geq 0}$. Note that, $2a+b\leq2+l$. 
\end{theorem}

\begin{remark}
    \label{gaussian-bounds-remark}
    The above result is stated on $\R^d$ in \cite{friedman2008partial}, and it is also mentioned that similar result holds true on bounded domains. Since, the result on $\T$ is not explained in detail in \cite{friedman2008partial}, we briefly comment on how to derive the Gaussian bounds on $\T$ by using the estimates on $\R^d$. Functions on $\T$ are identified with $\Z^d$-periodic functions on $\R^d$ in the standard way. Let $K$ and $K_{\R^d}$ be the fundamental solutions on $\T$ and $\R^d$, respectively. For $x \in \R^d$ and $y \in x+[-1/2,1/2)^d$
    \[K(x,t;y,s) = \sum_{k \in \Z^d} K_{\R^d} (x,t;y+k,s).\]
    Applying the Gaussian bounds of $K_{\R^d}$
    \[|K(x,t;y,s)| \leq C(t-s)^{-d/2}\sum_{k \in \Z^d}\exp \left( \frac{-c \abs{x-(y+k)}^{2}}{t-s} \right). \]
Next, we estimate  the sum above by splitting it into the cases $k=0$,
$0 < |k| \leq2\sqrt{d}$, and $|k| > 2\sqrt{d}$.  Recall that, $y \in
x+[-1/2,1/2)^d$, so $|x-y| \leq \sqrt{d}/2$. If $|k|>0$, then $|x-(y+k)| \geq c|x-y|$ for some dimensional constant $c$. On the other hand, if $|k| >  2\sqrt{d}$ then, 
    \[|x-(y+k)|^2 = |k|^2- 2(x-y) \cdot k+|x-y|^2  \geq  |k|^2 - 2\frac{\sqrt{d}}{2}|k| + |x-y|^2 \geq \frac{1}{2}|k|^2+ |x-y|^2.\]
    Thus, combining these estimates
   \begin{align*}
       |K(x,t;y,s)| &\leq C(t-s)^{-d/2}\exp \left( \frac{-c \abs{x-y}^{2}}{t-s} \right)\left[1+\sum_{|k| > 2 \sqrt{d}}\exp \left( \frac{-c \abs{k}^{2}}{t-s} \right) \right]\\
       &\leq C(t-s)^{-d/2}\exp \left( \frac{-c \abs{x-y}^{2}}{t-s} \right)\left[1+C(t-s)^p\sum_{|k| > 2 \sqrt{d}}|k|^{-2p} \right]\\
       & \leq C(t-s)^{-d/2}\exp \left( \frac{-c \abs{x-y}^{2}}{t-s} \right),
   \end{align*}
    as long as we take $p > d/2$ for summability. We have used the
    elementary inequality $e^{-r} \leq C(p)r^{-p}$ for any $p>0$ and
    all $r\geq 1$ with some $C(p)\geq 1$, as well as the finite time horizon assumption $t-s \leq 1$. Derivative estimates follow in a similar way.
\end{remark}
In Definition \ref{map-def} of the next section, via Duhamel's principle, we introduce an integral type map \eqref{map} associated with \eqref{FP} which includes the fundamental solution $K$ and its gradient so, it is beneficial to record some quick conclusions of the Gaussian bounds listed in Theorem \ref{gaussiann-bounds}.

\begin{corollary}[Integral bounds on the fundamental solution]\label{gaussian-integral-estimates}
    Let $\partial_t-L$ be a second-order uniformly parabolic operator in non-divergence form which satisfies the assumptions of Theorem \ref{gaussiann-bounds} for $l=1$. Then we have the following integral bounds on the derivatives of the fundamental solution $K$ of $\partial_t-L$:
    \begin{align}
        \label{double-int-fund-grad-bound}
        \int_{t'}^{t} \int_{\T}  \abs{\grad_y K(x,t;y,s)} dyds &\leq C \abs{t-t'}^{1/2}, \\
        \label{triple-int-fund-sol}
        \int_{0}^{t'} \int_{\T} \int_{t'}^{t} \abs{\partial_\tau \grad_y K(x',\tau;y,s)} d\tau dy ds &\leq C \abs{t-t'}^{1/2}, \\
        \intertext{and}
        \label{double-int-fund-space-holder}
        \int_{0}^{t} \int_{\T} \abs{\grad_y K(x,t;y,s)-\grad_y K(x',t;y,s)} dy ds &\leq C t^{(1-\beta)/2} \abs{x-x'}^{\beta},
    \end{align}
    where $\beta \in (0,1)$, $C>0$ is a constant as in Theorem \ref{gaussiann-bounds}, and $t\geq t'$.
\end{corollary}
\begin{proof}
    We only prove \eqref{triple-int-fund-sol}. \eqref{double-int-fund-grad-bound} and \eqref{double-int-fund-space-holder} are shown using similar techniques as \eqref{triple-int-fund-sol}. By using the identification $\T \cong [0,1)^d \subset \R^d$, we can estimate the space integral of Gaussian on $\T$ as
    \begin{align}
        \int_{\T} \exp \left(\frac{-c\abs{x'-y}^2}{\tau-s}\right)dy 
        \leq \int_{\R^d} \exp \left(\frac{-c\abs{x'-y}^2}{\tau-s}\right)dy 
        = C (\tau-s)^{d/2} \label{gaussian-int-on-torus}.
    \end{align}
    
    Now, we start to the estimation with an application of Tonelli's theorem in the left hand side of \eqref{triple-int-fund-sol}.
    \begin{align*}
        \int_{0}^{t'} \int_{\T} \int_{t'}^{t} &\abs{\partial_\tau \grad_y K(x',\tau;y,s)}d\tau dy ds \\ &= \int_{0}^{t'} \int_{t'}^{t} \int_{\T} \abs{\partial_\tau \grad_y K(x',\tau;y,s)} dy d\tau ds \\
        \overset{\eqref{gaussian-bound-time-space}}&{\leq} C \int_{0}^{t'} \int_{t'}^{t} \int_{\T} (\tau-s)^{-(d+3)/2} \exp \left(\frac{-c\abs{x'-y}^2}{\tau-s}\right) dy d\tau ds \\
        \overset{\eqref{gaussian-int-on-torus}}&{\leq} C \abs{\int_{0}^{t'} \int_{t'}^{t} (\tau-s)^{-3/2}  d\tau ds} \\
        &= C \abs{t^{1/2} - t'^{1/2} - (t-t')^{1/2}} \\
        &\leq C \abs{t-t'}^{1/2}
    \end{align*}
    In last step, we have used the elementary inequality $\abs{t^{1/2}-t'^{1/2}}\leq \abs{t-t'}^{1/2}$ which holds for all $t,t' \geq 0$.
\end{proof}
Before concluding this section, we list two Schauder estimates which will be used to obtain regularity for the fixed point solution of \eqref{FP}, see Corollary \ref{regularity} below. We mention that, although it is accessible by classical techniques, we could not find a classical reference for the first estimate \eqref{schauder-estimate-one}. Instead we are referring to a relatively recent work \cite{dong2015schauder}, which also has results in a much more general setting. In fact, the estimate \eqref{schauder-estimate-one} is quite practical since it allows us to obtain the first order spatial differentiability of the solution $u$ of
$\partial_t u = Lu + \grad \cdot A$ whenever $A$ is continuous. More precisely, we have the following result.

\begin{proposition}[\label{schauder-estimates}{Schauder estimates \cite[Prop.~4.1]{dong2015schauder}}]
    Let $\partial_t-L$ be a second-order uniformly parabolic operator in divergence form such that all of its coefficients are bounded and $\beta$-Hölder continuous. Also, let $Q_r$ denotes the parabolic cylinder of radius $r$ and $A(x,t) \in C_{x,t}^{\beta,0}(Q_1)$.
    If  $u \in C^{2+\beta,1+\beta/2}_{x,t}(Q_2)$ solves the PDE
    \begin{align*}
        \frac{\partial u}{\partial t} = Lu+\grad \cdot A(x,t) 
    \end{align*}
    then
    \begin{equation} 
        \label{schauder-estimate-one}
        \|u\|_{C^{1+\beta,(1+\beta)/2}_{x,t}(Q_1)} \leq C(\|u\|_{L^\infty(Q_2)} + [A]_{C^\beta_x(Q_2)}).
    \end{equation}
    Next, assume further that the coefficients of $\partial_t-L$ in non-divergence form are bounded and $\beta$-Hölder continuous. Let $B(x,t) \in C_{x,t}^{\beta,\beta/2}(Q_1)$, if $v \in C^{2+\beta,1+\beta/2}_{x,t}(Q_2)$ solves the PDE
    \begin{equation*}
        \frac{\partial v}{\partial t} = Lv+B(x,t)
    \end{equation*}
    then
    \begin{align}
        \label{schauder-estimate-two}
        \|v\|_{C^{2+\beta,1+\beta/2}_{x,t}(Q_1)} \leq C(\|v\|_{L^\infty(Q_2)} + \|B\|_{C^{\beta,\beta/2}_{x,t}(Q_2)}).
    \end{align}
    The constant $C$ depends only on the dimension, uniform parabolicity constant $\lambda$ of $L$, uniform upper bound and Hölder norms of the coefficients of $L$.
\end{proposition}

\section{Proof of Local Existence}
\label{local-existence-section}
In this section, we will establish the local existence of solutions of
\eqref{FP}. As presented in detail below, to prove the local existence, we intend to utilize certain tools from the linear parabolic PDEs and fixed point theory. 
With this aim in mind, we handle the logarithmic nonlinearity of
\eqref{FP} as a source term and derive a Duhamel type fixed point
formulation of the nonlinear PDE. After that, we study well-definiteness, contraction, and regularity properties of the fixed point mapping, defined below in \eqref{map}, to obtain the desired local existence result. 

In order to deal with the nonlinearity of \eqref{FP}, we introduce the following linear parabolic non-homogeneous equation which contains the nonlinearity of \eqref{FP} as a source term. The PDE is as follows:
\begin{equation}
    \tag{LPNE}
    \label{LPNE}
    \left\{
    \begin{aligned}
    & \frac{\partial u}{\partial t} = L_{\textup{FP}}u + \grad \cdot \left( \frac{\grad D (x)}{\pi (x,t)}f \log f \right), &\quad &x \in \T,\   t>0, \\
    & u(x,0)=f_0(x), &\quad &x \in \T,
    \end{aligned}   
    \right.
\end{equation}
where $L_{\textup{FP}}$ as defined in \eqref{parabolic-operator}. 
%
As we pointed out, \eqref{LPNE} is linear in the variable $u$, hence, by Duhamel's principle its solution can be written as:
\begin{equation}
    \tag{DP}
    \begin{aligned}
    \label{map-before-IBP}
    u = \int_{\T} &K(x,t;y,0)f_0(y)dy \\ &+\int_{0}^{t} \int_{\T} K(x,t;y,s) \grad_y \cdot \left( \frac{\grad D(y)}{\pi (y,s)}f(y,s) \log f(y,s)\right) dy ds,
    \end{aligned}
\end{equation}
where $K$ is the fundamental solution of the second-order uniformly
parabolic operator $\partial_t-L_{\textup{FP}}$. Here, $L_{\textup{FP}}$ as in \eqref{parabolic-operator} and its coefficients satisfy Assumptions \ref{assumption-constants} and \ref{regularity-of-coefficients}. 
The integral formulation of \eqref{LPNE} via \eqref{map-before-IBP} suggests a natural fixed point functional to prove existence via the Banach fixed point theorem.

Next we lay out the precise functional space setting where we will be able to apply the Banach fixed point theorem. In this regard, consider the Banach space
$X \coloneqq C^0(\T \times [0,T])$ which is equipped with the standard supremum norm

\begin{equation*}
    \normX{f} \coloneqq \sup_{(x,t)\in \T \times [0,T]} \abs{f(x,t)}
\end{equation*}
and $T>0$ specified in \eqref{time-bound} later on.
For our construction, besides the closed ball property, we additionally employ the positivity condition $f \geq \mu > 0$ and introduce the following norm-closed subset of $X$
\begin{equation}
    \label{set-Y}
    Y \coloneqq \{f \in X \ | \ f \geq \mu, \ \normX{f} \leq R \},
\end{equation}   
where  $\mu>0$ is the parameter from Assumption \ref{assumption-initial-data}, and $R \coloneqq 1 + \mu +  2\normC{f_0}$. 

First, we note that, the uniform upper and lower bounds---$R$ and $\mu$---in this functional setting provide uniform control on the regularity of the logarithmic nonlinearity, see the norm estimates \eqref{norm-estimate-one} and \eqref{norm-estimate-two} derived below.
%
Second, this positivity condition will also directly lead to existence of a positive solution for a given positive initial data (see Assumption \ref{assumption-initial-data}) which is physically relevant for our model.

Now, motivated by the formula \eqref{map-before-IBP}, via integration by parts, we introduce the following fixed point functional.

\begin{definition}[Map]\label{map-def}
    We define the map $\Psi: Y \rightarrow Y$, $f \mapsto u$ by the formula given below
    \begin{gather}
        \begin{aligned}\label{map}
        u = \Psi f (x,t) &= \int_{\T} K(x,t;y,0)f_0(y)dy \\ &\qquad-\int_{0}^{t} \int_{\T} \grad_y K(x,t;y,s) \cdot V(y,s) f(y,s) \log f(y,s) dy ds \\
        &\eqqcolon \Psi_{f_0,\textnormal{linear}}  (x,t) + \Psi_{f,\textnormal{nonlinear}} (x,t)
        \end{aligned}
        \tag{M}
    \end{gather}
    where $V(y,s) \coloneqq \frac{\grad D(y)}{\pi (y,s)}$.
\end{definition}

\begin{remark}
The integration by parts does not produce any boundary term since $\partial \T = \emptyset$. 
We use the subscripts ``linear" and ``nonlinear" since $\Psi_{f_0,\textnormal{linear}}$ solves the linear part of \eqref{FP} with the initial data $f_0$ and $\Psi_{f,\textnormal{nonlinear}}$ is associated with the nonlinear part of \eqref{FP}.
In the above form, it is easy to make use of Corollary \ref{gaussian-integral-estimates} along with the norm estimates \eqref{norm-estimate-one} and \eqref{norm-estimate-two} given below.
\end{remark}

Our nonlinear Fokker-Planck type equation \eqref{FP} contains logarithmic nonlinearity so, we should be able to control this term.
To this end, for any $f$ and $g$ in $Y \subset X$, we have the following norm estimates:
    \begin{align}
        \normX{\log f - \log g} &\leq \frac{1}{\mu} \normX{f-g} \label{norm-estimate-one} \\
        \intertext{and}
        \normX{\log f} &\leq \frac{R}{\mu} + \abs{\log \mu} + 1. \label{norm-estimate-two}
    \end{align}
    For the derivation, since both $f$ and $g$ are in $Y$, we have $f,g \geq \mu$ and by using $\frac{1}{\mu}$-Lipschitz continuity of the logarithm function on $[\mu,\infty)$ we have
    \begin{equation*}
        \abs{\log f - \log g} \leq \frac{1}{\mu}\abs{f-g}.
    \end{equation*}
    After that, by taking the supremum of both sides on $\T \times [0,T]$ we obtain \eqref{norm-estimate-one}. 
    For \eqref{norm-estimate-two}, by using the triangle inequality and Lipschitz continuity again, we can estimate
    \begin{equation*}
        \begin{split}
            \abs{\log f} &\leq \abs{\log f - \log \mu} + \abs{\log \mu} \\
                         &\leq \frac{1}{\mu} \abs{f-\mu} + \abs{\log \mu} \\
                         &\leq \frac{1}{\mu}\abs{f} + \abs{\log \mu} + 1.
        \end{split}
    \end{equation*}
    Then, similarly, by taking supremum of both sides on $\T \times [0,T]$ in the last inequality and using the fact that $\normX{f} \leq R$ gives \eqref{norm-estimate-two}. 


First, recall $R = 1 + \mu + 2\normC{f_0}$, $\mu>0$ is arbitrary but fixed real number as in  Assumption \ref{assumption-initial-data}, $C>0$ is a constant depending only on the constants in Theorem \ref{gaussiann-bounds} which will be fixed during the proof, and $V=\frac{\grad D}{\pi}$ as in Definition \ref{map-def}. Also, $W_{\textup{inf}}$ and $W_{\textup{sup}}$ denotes the infimum and supremum of the coefficient $W(x,t) \coloneqq \grad \cdot \left( \frac{\grad \phi (x)}{\pi (x,t)} \right)$ of the zeroth order term in $\partial_t-L_{\textup{FP}}$. Then we choose the time bound given below:

\begin{equation}\label{time-bound}
    T^{1/2} \leq \min \left\{ \frac{\min \{ \mu,1 \} }{ 2 \left( CR \left( \frac{2R}{\mu} + \abs{\log \mu} + 1 \right) \normV{V} +1 \right)}, \left( \frac{\log 2}{\abs{W_{\textup{inf}}}+\abs{W_{\textup{sup}}}+1} \right)^{1/2} \right\}.
\end{equation}

With these preparations, in the next lemma, we show that the map \eqref{map} is well-defined, i.e. maps $Y$ to itself. Hereafter, in proofs, we adapt the notational conventions from Section \ref{notation} for brevity.

\begin{lemma}[Well-definiteness of the map]\label{well-definedness}
    The map $\eqref{map}$ is well-defined i.e. maps $Y$ to itself where $Y$ as defined in \eqref{set-Y}.
\end{lemma}
\begin{proof}
    First, we should prove that, for any $f \in Y$, $(\Psi f)(x,t)$ defines a continuous function of $(x,t)$. From the definition of map \eqref{map}, we have the decomposition
    \begin{equation*}
        (\Psi f)(x,t) = (\Psi_{f_0,\textnormal{linear}})(x,t) + (\Psi_{f_,\textnormal{nonlinear}})(x,t)
    \end{equation*}
    so, it is sufficient to prove the continuity of the each part separately. The linear part, $(\Psi_{f_0,\textnormal{linear}})(x,t)$, is continuous function of $(x,t)$ according to Theorem \ref{rep-formula}. 

    For the nonlinear part, we start with the following difference:
    \begin{align*}
        (\Psi&_{f,\textnormal{nonlinear}})\left(x',t'\right) - (\Psi_{f,\textnormal{nonlinear}})(x,t) \\
        &=\int_{0}^{t} \int_{\T} \grad_y K(x,t;y,s) \cdot  V(y,s)f(y,s) \log f(y,s) dy ds \\
        &\quad-\int_{0}^{t'} \int_{\T} \grad_y K(x',t';y,s) \cdot V(y,s)f(y,s) \log f(y,s) dy ds \\
        &=\int_{0}^{t} \int_{\T} (\grad_y K(x,t;y,s) - \grad_y K(x',t;y,s)) \cdot V(y,s)f(y,s) \log f(y,s) dy ds \\
        &\quad+\int_{0}^{t'} \int_{\T} (\grad_y K(x',t;y,s)-\grad_y K(x',t';y,s)) \cdot V(y,s)f(y,s) \log f(y,s) dy ds \\
        &\quad \quad + \int_{t'}^{t} \int_{\T} \grad_y K(x',t;y,s) \cdot V(y,s)f(y,s) \log f(y,s) dy ds \\
        &\eqqcolon I_1 + I_2 + I_3.
    \end{align*}
    The next step is estimation of the integrals $I_1,I_2$, and $I_3$. After that, we will combine our estimates on $I_1,I_2$, and $I_3$ to obtain desired continuity result.\\

    \noindent \textit{Estimation of $I_1$}: 
    \begin{align*}
        \abs{I_1} &\leq \int_{0}^{t} \int_{\T} \abs{\grad_y K(x,t;y,s) - \grad_y K(x',t;y,s)} \abs{V(y,s)} \abs{f(y,s)} \abs{\log f(y,s)} dy ds \\
        &\leq \normX{f}\normX{\log f} \normV{V} \int_{0}^{t} \int_{\T} \abs{\grad_y K(x,t;y,s)-\grad_y K(x',t;y,s)} dy ds \\
        \overset{\eqref{double-int-fund-space-holder},\eqref{norm-estimate-two}}&{\leq} CR T^{(1-\beta)/2}\left(\frac{R}{\mu}+\abs{\log \mu}+1\right) \normV{V} \abs{x-x'}^{\beta}
    \end{align*}
    \textit{Estimation of $I_2$}:
    \begin{align*}
        \abs{I_2} &\leq \int_{0}^{t'} \int_{\T} \abs{\grad_y K(x',t;y,s)-\grad_y K(x',t';y,s)} \abs{V(y,s)} \abs{f(y,s)} \abs{\log f(y,s)} dy ds \\
        &\leq \normX{f}\normX{\log f} \normV{V} \int_{0}^{t'} \int_{\T} \abs{\grad_y K(x',t;y,s)-\grad_y K(x',t';y,s)} dy ds \\
        \overset{\textrm{FTC},\eqref{norm-estimate-two}}&{\leq} R\left(\frac{R}{\mu}+\abs{\log \mu}+1\right) \normV{V} \int_{0}^{t'} \int_{\T} \int_{t'}^{t} \abs{\partial_\tau \grad_y K(x',\tau;y,s)}d\tau dy ds \\
        \overset{\eqref{triple-int-fund-sol}}&{\leq} CR\left(\frac{R}{\mu}+\abs{\log \mu}+1\right) \normV{V} \abs{t-t'}^{1/2}
    \end{align*}
    \textit{Estimation of $I_3$}:
    \interdisplaylinepenalty=2500
    \begin{align*}
        \abs{I_3} &\leq \int_{t'}^{t} \int_{\T} \abs{\grad_y K(x',t;y,s)} \abs{V(y,s)} \abs{f(y,s)} \abs{\log f(y,s)} dy ds \\
        &\leq \normX{f} \normX{\log f} \normV{V} \int_{t'}^{t} \int_{\T} \abs{\grad_y K(x',t;y,s)} dy ds \\
        \overset{\eqref{double-int-fund-grad-bound},\eqref{norm-estimate-two}}&{\leq} CR\left(\frac{R}{\mu}+\abs{\log \mu}+1\right) \normV{V} \abs{t-t'}^{1/2}
     \end{align*}
    By combining estimates of $I_1,I_2$, and $I_3$, for some constant $C$, we obtain 
    \begin{equation}\label{space-hold-cont-nonlin-pt}
        \abs{(\Psi_{f,\textnormal{nonlinear}})(x,t) - (\Psi_{f,\textnormal{nonlinear}})\left(x',t'\right)} \leq C \left(\abs{x-x'}^\beta + \abs{t-t'}^{1/2} \right)
    \end{equation}
    which proves the continuity of the nonlinear part. Thus, $(\Psi f)(x,t)$ is continuous function of $(x,t)$ whenever $f \in Y$.

    Second, we want to show that $\normX{\Psi f} \leq R$ whenever $\normX{f} \leq R$. We can estimate
    \begin{align*}
        \abs{\Psi f} &\leq \int_{\T} \abs{K(x,t;y,0)} \abs{f_0(y)}dy \\  &\quad + \int_{0}^{t} \int_{\T} \abs{\grad_y K(x,t;y,s)}  \abs{V(y,s)}\abs{f(y,s)} \abs{\log f(y,s)} dy ds \\
        \overset{\eqref{positivity-fund-sol}}&{\leq} \normC{f_0} \int_{\T} K(x,t;y,0)dy \\ & \quad + \normX{f}\normX{\log f} \normV{V} \int_{0}^{t} \int_{\T} \abs{\grad_y K(x,t;y,s)} dy ds \\
        \overset{\eqref{single-int-fund-bound},\eqref{double-int-fund-grad-bound},\eqref{norm-estimate-two}}&{\leq} \exp(W_\textup{sup}t) \normC{f_0} + CRT^{1/2} \normV{V} \left( \frac{R}{\mu} + \abs{\log \mu} + 1 \right) \\
        \overset{\eqref{time-bound}}&{\leq} 2\normC{f_0} + 1 \\
        &\leq R.
    \end{align*}
    Then, by taking supremum on $\T \times [0,T]$ in the both sides of the last inequality, we obtain $\normX{\Psi f} \leq R$. 

    Third, we should lastly show that $\Psi f \geq \mu$ whenever $f \geq \mu$. By using Assumption \ref{assumption-initial-data} i.e. $f_0(x) \geq 4\mu$ 
    together with strict positivity and comparison principle properties of the fundamental solution $K$, we can estimate
    \begin{align*}
        \Psi f (x,t) &= \int_{\T} K(x,t;y,0)f_0(y)dy \\  &\quad-\int_{0}^{t} \int_{\T} \grad_y K(x,t;y,s) \cdot V(y,s)f(y,s) \log f(y,s) dy ds \\
        \overset{\textup{\ref{assumption-initial-data}},\eqref{positivity-fund-sol}}&{\geq} 4\mu \int_{\T} K(x,t;y,0)dy \\&\qquad \qquad-\int_{0}^{t} \int_{\T} \abs{\grad_y K(x,t;y,s)} \abs{V(y,s)}\abs{f(y,s)} \abs{\log f(y,s)} dy ds \\
        \overset{\eqref{single-int-fund-bound}}&{\geq} 4\mu \exp(W_\textup{inf} t) - \normX{f} \normX{\log f} \normV{V} \int_{0}^{t} \int_{\T} \abs{\grad_y K(x,t;y,s)} dyds \\
        \overset{\eqref{double-int-fund-grad-bound},\eqref{norm-estimate-two}}&{\geq} 4\mu \exp(-\abs{W_\textup{inf}} T) - CRT^{1/2} \normV{V} \left( \frac{R}{\mu} + \abs{\log \mu} + 1 \right) \\
        \overset{\eqref{time-bound}}&{\geq} 2\mu - \min \{ \mu,1 \} \\
        &\geq \mu .
    \end{align*}
    Thus, we obtain $\Psi f \geq \mu$ whenever $f \geq \mu$. These are show map \eqref{map} is well-defined on $Y$ and complete the proof.
\end{proof}

Following an estimation procedure similar to the proof of the previous lemma, in the next lemma, we establish the contraction property of the map \eqref{map}.
\begin{lemma}[Contraction property of the map]\label{contraction-property}
    The map \eqref{map} is $\frac{1}{2}$-contraction mapping on $Y$.
\end{lemma}

\begin{proof}
    Assume that $f$ and $g$ are two arbitrary functions from $Y$. We estimate the following difference in the map \eqref{map}:
    \begin{align*}
         |&(\Psi f)(x,t) - (\Psi g)(x,t)| \\
         &\leq \int_{0}^{t} \int_{\T} |\grad_y K(x,t;y,s)| \abs{ V(y,s)\left(f(y,s) \log f(y,s) - g(y,s)\log g(y,s) \right) } dy ds \\ 
         &= \int_{0}^{t} \int_{\T} |\grad_y K(x,t;y,s)| \biggl| V(y,s)(f(y,s)\log f(y,s) - g(y,s)\log f(y,s) \\ 
         &\qquad \qquad \quad + g(y,s)\log f(y,s) - g(y,s)\log g(y,s) )\bigg|  dy ds \\
         &\leq ( \normX{f-g}\normX{\log f} + \normX{g} \normX{\log f - \log g} ) \normV{V} \int_{0}^{t} \int_{\T} |\grad_y K(x,t;y,s)|  dy ds \\
         &\overset{\eqref{double-int-fund-grad-bound},\eqref{norm-estimate-one},\eqref{norm-estimate-two}}{\leq} C T^{1/2} \left( \frac{2R}{\mu} + \abs{\log \mu} + 1 \right) \normV{V} \normX{f-g} \\
         &\overset{\eqref{time-bound}}{\leq} \frac{1}{2} \normX{f-g}.
    \end{align*}
    Finally, by taking supremum of both sides in the last inequality on $\T \times [0,T]$, we obtain the contraction property of the map \eqref{map} and this completes the proof.
\end{proof}

Having at hand, the well definiteness and contraction properties of the map \eqref{map}, we can obtain the following existence result now.

\begin{corollary}[Existence of a fixed point]\label{existence-of-a-fixed-point}
    There exists a unique fixed point $f$ of the map \eqref{map}.    
\end{corollary}
\begin{proof}
    From Lemmas \ref{well-definedness} and \ref{contraction-property}, we know that the map \eqref{map} is well-defined contraction mapping on $Y$ so, by Banach's fixed point theorem, there exists a unique fixed point $f \in Y$ such that $f = \Psi f$ and this completes the proof.
\end{proof}

In \eqref{space-hold-cont-nonlin-pt} we get something stronger which means that $\Psi_{f,\textnormal{nonlinear}}$ is $\beta$-Hölder continuous in space and $\frac{1}{2}$-Hölder continuous in time. From Theorem \ref{rep-formula}, we also know that $\Psi_{f_0,\textnormal{linear}}$ is $C^{2,1}_{x,t}$. In particular, we conclude that $ f= \Psi_{f_0,\textnormal{linear}} + \Psi_{f,\textnormal{nonlinear}}$ is $\beta$-Hölder continuous in space.  
By using this information, in the next corollary, we prove that $f$ is in Hölder space $C^{2+\beta,1+\beta/2}_{\textup{loc}}$ to get a classical solution of \eqref{FP}.

\begin{corollary}[Regularity]\label{regularity}
    The unique fixed point $f$ of \eqref{map} belongs to
    \begin{equation*}
        f \in C^{2+\beta,1+\beta/2}_{\textup{\textup{loc}}}(\T \times (0,T])
    \end{equation*}
    where $T$ as characterized in \eqref{time-bound}.
    So, this means $f$ is two times differentiable in $x$ and one time differentiable in $t$ with Hölder continuous derivatives. Hence, $f$ is the unique classical solution of \eqref{FP}.
\end{corollary}
\begin{proof}
    We basically adopt the argument of Krylov \cite[Thm.~8.7.3]{krylov1996lectures} together with Schauder estimates \eqref{schauder-estimate-one} and \eqref{schauder-estimate-two} to our problem. 
    From the discussion above, $f$ is $\beta$-Hölder continuous in space and $V=\frac{\grad D}{\pi}$ is also $\beta$-Hölder continuous in space due to Assumption \ref{regularity-of-coefficients}. As a conclusion, $F \coloneqq V f \log f$ is $\beta$-Hölder continuous in space. Moreover, $L_{\textup{FP}}$ as in \eqref{parabolic-operator} is second-order uniformly parabolic operator in divergence form with $\beta$-Hölder continuous coefficients due to  Assumptions \ref{assumption-constants} and \ref{regularity-of-coefficients}. 
    
    Since $F \in C_x^\beta$, there is a sequence of functions $F_n \in C^\infty_{\textup{loc}} \intersect C^{\beta}_x$ which converge uniformly to $F$ and have uniformly bounded $C^{\beta}_x$-norms. Let $u_n$ be the corresponding solutions of \eqref{LPNE} with forcing $F_n$. Since $u_n \in C^{2+\beta,1+\beta/2}_{x,t}$ by \cite[Thm.~10,~Ch.~3]{friedman2008partial} Schauder estimate \eqref{schauder-estimate-one} applies, and therefore the $C^{1+\beta,(1+\beta)/{2}}_{x,t}$-norms of $u_n$ are uniformly bounded. Then one can establish that $u_n \rightarrow f$ uniformly, by using the uniform convergence of the $F_n$ and the representation formula \eqref{map}. By a standard argument with distributional derivatives we can derive that $\|f\|_{C^{1+\beta, (1+\beta)/2}_{x,t}} \leq \liminf \|u_n\|_{C^{1+\beta, (1+\beta)/2}_{x,t}}$ \cite[Ex.~8.5.6]{krylov1996lectures}. 
    Since $V$ is also in $C^{1+\beta,\beta/2}_{x,t}$ due to Assumption \ref{regularity-of-coefficients}, this regularity of $f$ makes the divergence term $\grad \cdot F \in C^{\beta,\beta/2}_{x,t}$. Therefore, by applying the non-divergence form Schauder estimate \eqref{schauder-estimate-two} and similar regularization procedure it is possible to upgrade the regularity of $f$ to $C^{2+\beta,1+\beta/{2}}_{\textup{loc}}$ and this completes the proof. 
\end{proof}
\begin{remark}[Further regularity]
    In addition to boundedness, if we also assume that the
    coefficients listed in Assumption \ref{regularity-of-coefficients}
    are smooth, i.e. they belong to $C^\infty(\T \times [0,\infty))$,
    then the fundamental solution $K$ of the operator $\partial_t - L_{\textup{FP}}$ will also be smooth, thanks to \cite[Thm.~9,~Ch.~9]{friedman2008partial}. Thus, we can inductively repeat the regularization procedure given in Corollary \ref{regularity} to conclude the smoothness of the solution of \eqref{FP}.
\end{remark}

What happens if the initial data $f_0$ and $g_0$ are close to each other? In the next lemma, we answer this question.
\begin{corollary}[Continuous dependence on initial data]\label{continuous-dependence}
    Let $f$ and $g$ be solutions of \eqref{FP} with initial data $f_0$ and $g_0$ respectively. Then the following continuity estimate holds
    \begin{equation}\label{continuity-estimate}
        \normX{f-g} \leq 4 \normC{f_0-g_0}.
    \end{equation}
\end{corollary}
\begin{proof}
    For the fixed points $f$ and $g$ of the map \eqref{map} with initial data $f_0$ and $g_0$ respectively, we can estimate
    \begin{align*}
         |f - g| &\leq \int_{\T} \abs{K(x,t;y,0)} \abs{f_0(y)-g_0(y)}dy \\
         & \  + \int_{0}^{t} \int_{\T} |\grad_y K(x,t;y,s)| \abs{ V(y,s)\left(f(y,s) \log f(y,s) - g(y,s)\log g(y,s) \right)} dy ds \\
         &\leq 2\normC{f_0-g_0} + \frac{1}{2} \normX{f-g}
    \end{align*}
    where we have used \eqref{positivity-fund-sol}, \eqref{single-int-fund-bound}, \eqref{time-bound}, and Lemma \ref{contraction-property} in the last step. Then, by taking supremum of both sides on $\T \times [0,T]$ in the last inequality and then by doing re-arrangement, we arrive to \eqref{continuity-estimate} which completes the proof.
\end{proof}

As a consequence of lemmas and corollaries that we proved in this section, we arrive to the final result of this section, which is on local well-posedness.
\begin{theorem}[Local well-posedness]\label{local-well-posedness}
    Under the assumptions \ref{assumption-constants}, \ref{regularity-of-coefficients}, and \ref{assumption-initial-data} the initial value problem \eqref{FP} is locally well posed on $Y$.
\end{theorem}    
\begin{proof} 
    Follows from Corollaries \ref{existence-of-a-fixed-point}, \ref{regularity}, and \ref{continuous-dependence}.
\end{proof}

\begin{remark}[Generalization of \eqref{FP}]\label{generalization-of-FP}
    For any locally Lipschitz function $N(f)$ on $(0,\infty)$ and a vector field $G(x,t) \in C^{1+\beta,\beta/2}_{x,t}$,
    under the similar assumptions, it is possible to generalize all of our analysis to the following class of PDEs
    \begin{equation*}
        \left\{
        \begin{aligned}
        & \frac{\partial f}{\partial t} = Lf + \grad \cdot \left( G(x,t) N(f) \right), &\quad &x \in \T,\   t>0, \\
        & f(x,0)=f_0(x), &\quad &x \in \T,
        \end{aligned}   
        \right.
    \end{equation*}
    since Gaussian bounds in Theorem \ref{gaussiann-bounds} work for any second-order uniformly parabolic differential operator $\partial_t-L$ 
    with Hölder continuous coefficients.
\end{remark}
\label{local-existence-proof}

\section{Proof of Global Existence}
\label{global-existence-section}
In this section, we will show the existence of the unique global solution of \eqref{FP}. Our argument for obtaining the unique global solution relies on repeating our local existence proof on nested time intervals and using an a priori estimate on the maximum and minimum of the ratio between a classical solution $f$ and the equilibrium state $f^{\mathrm{eq}}$ (see \eqref{eq-state}). 

Recall the following a-priori estimates, proved in \cite{epshteyn2022nonlinear,epshteyn2024longtimeasymptoticbehaviornonlinear}.

\begin{proposition}[A priori estimates {\cite[Prop.~1.6]{epshteyn2022nonlinear}}, {\cite[Cor.~1.8]{epshteyn2024longtimeasymptoticbehaviornonlinear}}]\label{a-priori-estimate-1}
    Let $f(x,t)$ be a classical solution of \eqref{FP}. Then for all $(x,t) \in \T \times (0,\infty)$, under the assumptions \ref{assumption-constants}, \ref{regularity-of-coefficients}, \ref{assumption-initial-data}, and \ref{assumption-boundedness}, the following two sided estimate holds on $f(x,t)$:
    \begin{align*}
         \exp &\left( \frac{1}{D(x)} \min_{y\in\T} \left( D(y)  \log \frac{f_0(y)}{f^{\mathrm{eq}}(y)} \right) \right) f^{\mathrm{eq}}(x) \\
         &\leq f(x,t) \leq \exp \left( \frac{1}{D(x)} \max_{y\in\T} \left( D(y)  \log \frac{f_0(y)}{f^{\mathrm{eq}}(y)} \right) \right) f^{\mathrm{eq}}(x).
    \end{align*}
    Also, by using the upper and lower bounds of $f^{\textup{eq}}$ from Lemma 1.6 of \cite{epshteyn2024longtimeasymptoticbehaviornonlinear}, it is possible to make it the above estimate uniform:
    \begin{equation} \label{a-priori-estimate}
        m \leq f(x,t) \leq M
    \end{equation} 
    where $m,M>0$ are constants which are independent from $x$ and $t$.
\end{proposition}

\begin{theorem}[Global well-posedness]\label{global-well-posedness}
    Under the assumptions \ref{assumption-constants}, \ref{regularity-of-coefficients}, \ref{assumption-initial-data}, and \ref{assumption-boundedness}  
    there exists a unique positive classical global solution of \eqref{FP} which belongs to
    \begin{equation*}
        f \in C^{2+\beta,1+\beta/2}_{\textup{loc}}(\T \times (0,\infty)).  
    \end{equation*}
\end{theorem}

\begin{proof}
We will construct the solution via induction. Define
\[ R' \coloneqq 1 + \mu + 2\normC{f_0} + 2M\]
and
\[
     T'^{1/2} \coloneqq \min \left\{ \frac{\min \{ \mu,1,\frac{m}{4} \} }{ 2 \left( CR' \left( \frac{2R'}{\gamma} + \abs{\log \gamma} + 1 \right) \normV{V} +1 \right)}, \left( \frac{\log 2}{\abs{W_{\textup{inf}}}+\abs{W_{\textup{sup}}}+1} \right)^{1/2} \right\} \label{time-bound-2}
\]
where $C>0$ is a constant as in Theorem \ref{gaussiann-bounds} and $\gamma \coloneqq \min \{ \mu, \frac{m}{4} \}$. The claim is: for all $j \in \mathbb{Z}_{\geq 1}$, there is a classical solution of \eqref{FP} on $[0,jT']$ satisfying $ m \leq f \leq M$.

The base case, $j=1$, similarly follows from the arguments that we presented in Section \ref{local-existence-section} with the updated values of $R'$ and $T'$. The bounds $ m \leq f \leq M$ follow from the a-priori estimate in Proposition \ref{a-priori-estimate-1}.

Assume now that there exists a classical solution $f_k$ of \eqref{FP} on $[0,kT']$ satisfying $ m \leq f_k \leq M$. Define the following map on the time interval $[kT',(k+1)T']$
\begin{align*}\label{solution-map-k}
    \Phi g (x,t) \coloneqq \int_{\T} &K(x,t;y,kT')f_k(y,kT')dy 
    \\ &-\int_{kT'}^{t} \int_{\T} \grad_y K(x,t;y,s) \cdot \frac{\grad D(y,s)}{\pi(y,s)} g(y,s) \log g(y,s) dy ds
\end{align*}
on the following space
\begin{equation*}\label{set-Y-k}
    Y_k \coloneqq \Big\{g \in C^0(\T \times [kT',(k+1)T']) \eqqcolon X_k \ \big| \ g \geq \min \Big\{ \mu, \frac{m}{4} \Big\}, \ \|g \|_{X_k} \leq R' \Big\}.
\end{equation*}
Since, by the inductive hypothesis, we have $f_k(y,kT') \geq m$ so, we can argue exactly as in Lemma~\ref{well-definedness}, with $m$ as the lower bound on the initial data now instead of $4 \mu$, to show that $\Phi$ maps $Y_k$ to itself. The other parts of the local existence proof are the same as in the previous section and we conclude the existence of a fixed point solution of \eqref{FP} $g_{k+1}=\Phi g_{k+1} \in C^{2+\beta,1+\beta/{2}}_{x,t}(\T \times [kT',(k+1)T'])$. By construction, $f_k(y,kT')=g_{k+1}(y,kT')$ so the concatenation
\[f_{k+1}(y,t):= \begin{cases}
    f_k(y,t), &\text{if }~ 0 \leq t \leq kT'\\
    g_{k+1}(y,t), &\text{if }~ kT' \leq t \leq (k+1)T'
\end{cases}\]
is a fixed point solution of \eqref{FP} on $[0,(k+1)T']$, and therefore, by Corollary~\ref{regularity}, it is a classical solution of \eqref{FP} on $[0,(k+1)T']$. Consequently, the a-priori estimate \eqref{a-priori-estimate} applies and $m \leq f_{k+1} \leq M$.  This establishes the inductive hypothesis for $j=k+1$. By induction we conclude the existence of a positive global classical solution $f$ of \eqref{FP} in $C^{2+\beta,1+\beta/2}_{\textup{loc}}(\T \times (0,\infty))$. 

Moreover, uniqueness follows by a standard argument, since the set of times where two classical solutions are identical is open by another application of Banach's fixed point theorem on the same type of spaces defined above. 
\end{proof}
\label{global-existence-proof}

\section*{Acknowledgments}
We are grateful to the anonymous referees for their valuable comments and suggestions that helped to improve the paper. We also would like to thank Dr. Quoc-Hung Nguyen who pointed out that Theorem \ref{gaussiann-bounds} was missing the finite time horizon assumption in the initial preprint posted in arXiv.

The work of B. Bayir and Y. Epshteyn was partially supported by the NSF grant DMS-2118172, and the work of Y. Epshteyn was also partially supported by the Simons Foundation Fellowship SFI-MPS-SFM-00010667. The work of W. M. Feldman was partially supported by the NSF grant DMS-2407235.

\bibliographystyle{siamplain}
\bibliography{bibliography}

\end{document}